\documentclass[12pt, a4paper]{amsart} 
\usepackage[T2A]{fontenc}
\usepackage[utf8]{inputenc}
\usepackage[dvips]{graphicx}
\usepackage{geometry}
\usepackage{amsmath, amssymb, amsfonts, amsthm}
\usepackage{enumitem}
\usepackage{graphicx}
\usepackage{tikz}
\usepackage{hyperref}
 \usepackage{dutchcal}
\usepackage{bbm}

\usepackage{comment}
\usepackage[utf8]{inputenc}
\usepackage[T2A]{fontenc}
\usepackage[english]{babel}
\usepackage{eulervm}
\newtheorem{theorem}{Theorem}
\newtheorem{proop}{Proposition}
\newtheorem{prop}{Proposition}

\newtheorem{theom}{Theorem}

\newtheorem{lemma}{Lemma}
\newtheorem{defin}{Definition}
\newtheorem{probl}{Problem}
\theoremstyle{remark}
\newtheorem*{rem}{Remark}

\newenvironment{proof}{\noindent{\bf Proof:}}{$\hfill \Box$ \vspace{10pt}}  

\def\XXint#1#2#3{{\setbox0=\hbox{$#1{#2#3}{\int}$ }
\vcenter{\hbox{$#2#3$ }}\kern-.6\wd0}}

\newcommand{\hel} {
\hskip2.5pt{\vrule height7pt width.5pt depth0pt}
\hskip-.2pt\vbox{\hrule height.5pt width7pt depth0pt}
\, }
\newcommand{\restr}{\hel}
\newcommand{\vertiii}[1]{{\left\vert\kern-0.25ex\left\vert\kern-0.25ex\left\vert #1 
    \right\vert\kern-0.25ex\right\vert\kern-0.25ex\right\vert}}
    
\newcommand{\upin}{\rotatebox[origin=c]{95}{\(\in\)}}
    
\begin{document}

\title[Busemann--Hausdorff densities of dimension and codimension two]{On Busemann--Hausdorff densities of dimension two and of codimension two, with an application to the Plateau Problem}
\author{Ioann Vasilyev}
%Linear planes in Minkowski spaces: lack of area conractions and minimal sections of unit balls.
%\thanks{The author was supported by the Russian Science Foundation (grant No.~18-11-00053).}
%\address{Universit\'e Paris-Saclay, LAMA (UMR 8050), 61 avenue du G\'en\'eral de Gaulle, 94010, Cr\'eteil, France}
\address{CNRS-Laboratoire AGM-Universit\'e de Cergy-Pontoise, France.
\newline
\textit{On leave from: St.-Petersburg Department of V.A. Steklov Mathematical Institute, Russian Academy of Sciences (PDMI RAS), Fontanka 27, St.-Petersburg, 191023, Russia}}
%\email{milavas@mail.ru}
\email{ioann.vasilyev@cyu.fr, milavas@mail.ru}

\begin{abstract}
The purpose of this paper is twofold.  First, we describe one (presumably) new case, in which Busemann--Hausdorff densities are convex. We apply the corresponding result to prove the existence of minimizing rectifiable chains of codimension two in complex finite dimensional normed vector spaces.  Second, we prove that for each $n\geq 4$, there exists an $n$ dimensional normed vector space 
%the spaces $\ell_1^n$,  
in which the corresponding two  dimensional Busemann--Hausdroff density is not totally convex. %This completes the main result of the paper~\cite{buriv} by D. Burago and S. Ivanov, who proved that the densities are always extendibly convex in the two dimensional case. 
 This gives 
%our main result gives %partial 
 a negative answer to a question posed by %Ambrosio and Schmidt and 
 H. Busemann and E. Strauss, see~\cite{busemann2}.
\end{abstract}

%% Beginn des Briefes %%%%%%%%%%%%%%%%%%%%%%%%%%%%%%%%%%%%%%%
\maketitle
\section{Introduction}
This paper mostly concerns geometry of finite dimensional normed vector spaces. We point out from the very beginning  that in each finite dimensional normed vector space $E$, discussed in the current paper, we fix an Auerbach basis and the corresponding dual basis in the dual space $E^*$. We shall sometimes call these bases canonical. Also, we fix in $E$ a Euclidean structure, which is a choice of a symmetric, positive definite, real valued quadratic form defined on $E$. Natural objects related to this Euclidean structure (such as: sets, norms, metrics, measures, etc.) will be referred to as Euclidean. 

Let $E$ be a finite dimensional normed vector space. Recall that, for a natural number $m$ satisfying $1\leq m\leq \dim(E)$, the Grassmann cone of dimension $m$ consists of all simple (also called decomposable) $m$ vectors and is denoted $G(m,E)$. The corresponding $m$'th exterior power is denoted $\Lambda_m(E)$; it naturally inherits the Euclidean norm $|\cdot|_2$ and the inner product $\langle \cdot,\cdot \rangle$ of the ambient space $E$. The $m$'th Grassmannian of $E$ is the set that consists of all $m$ dimensional linear subspaces in $E$ and is denoted $Gr(m,E)$. The elements of $Gr(m,E)$ will be called $m$ planes. For the sake of brevity, when $m=2$, we shall sometimes say  ``planes'' instead of ``two planes''. Given $\omega\in G(m,E)$ with $\omega=w_1\wedge \ldots\wedge w_m$ for some linearly independent vectors $w_1,\ldots, w_m$ in $E$, we shall denote by $\mathrm{span}(\omega)$ the corresponding element of $Gr(m,E)$, i.e. $\mathrm{span}(w_1,\ldots,w_m)$. 

	The following property of functions on Grassmann cones is classical. 
	%\begin{defin}{(Extendibly convex functions)}
	 Let $E$ be an $n$ dimensional normed vector space and let $\phi$ be a function defined on the Grassmann cone  $G(m,E)$. We say that $\phi$ is extendibly convex if it is the restriction of a norm defined on $\Lambda_m(E)$.
	%	\end{defin}
Note that when $m$ is different from either $1$ or $n-1$, then the corresponding set $G(m,E)$ is not convex, which motivates the definition above. On the other hand, in the cases when $m=1$ or $m=n-1$, the notion of extendible convexity coincides with the usual convexity. Extendibly convex functions were extensively studied by H. Busemann and his school in the articles~\cite{busemann},~\cite{busemann1},~\cite{busemann2} and~\cite{phadke}.

We would like to remind the reader of two important types of functions defined on Grassmann cones.
%\begin{defin}{(Density and volume density)}
	Let  $E$ be a finite dimensional normed vector space of dimension $n$ and let $1\leq m\leq n$. An $m$ density function (or simply an $m$ density) is a continuous function $\phi: G(m,E)\rightarrow \mathbb R_+$ that is positively homogeneous of degree one. An $m$ density $\phi$ is an $m$ volume density if for all $e\in G(m,E)$ one has $\phi(e)\geq 0$ with equality if and only if $e=0$.
	%\end{defin}
	
\begin{rem}%\label{volume}
	Let  $E$ and $m$ be as in the previous definition and let  $\phi$ be an $m$ volume density on $E$. Then, for an $m$ rectifiable subset (see~\cite{federer} for a definition and related notions) $S\subset X$ one defines the corresponding volume $\text{Vol}_{\phi}(S)$ by the following formula
$$\text{Vol}_{\phi}(S):=\int_{S}\phi(T_xS)d\mathcal H^m(x),$$
where $T_xS$ stands for the corresponding approximate tangent $m$ plane and $\mathcal H^m$ is the Euclidean Hausdorff measure. 	
\end{rem} 

Recall one classical volume density, since it will be very important for us in what follows.
\begin{defin}{(Busemann--Hausdorff density)}
\label{bhd}
	Let  $E$ be an $n$ dimensional normed vector space with the unit ball $B$. For each $1\leq m \leq n$ we define the $m$ Busemann--Hausdorff density  $\phi_{m,BH,B}: G(m,E)\rightarrow \mathbb R_+$ by the following formula
	$$\phi_{m,BH,B}(\cdot)=\frac{\mathbold{\alpha}(m) |\cdot|_2}{\mathcal H^m(B\cap \mathrm{span}(\cdot))},$$
	where %$W\in G(m,X)$ and 
$\mathbold{\alpha}(m):=\mathbold{\pi}^{m/2}/\Gamma(m/2+1)$.  We shall sometimes write $BH$ densities to save space. For the same reason, we shall also regularly omit the third subscript $B$, especially when the choice of the ambient space $E$ (and hence also of its unit ball $B$) is evident. 
	\end{defin}
	
	\begin{rem}
 It is well known that the volume, that corresponds to a Busemann--Hausdorff density coincides with the corresponding intrinsic Hausdorff measure in $E$, i.e. with the one induced by the norm of $E$. 
\end{rem}

\bigskip
	
	In $1949$, in his famous article~\cite{busemann}, H. Busemann proved the following result concerning the so-called cross-sections of convex bodies in finite dimensional normed vector spaces.
	\begin{theom}
	\label{thm3.1}
	Let  $E$ be an $n$ dimensional normed vector space with the unit ball $B$. Then the function  $\phi_{n-1,BH,B}$ is extendibly convex.
	\end{theom}
In his proof of Theorem~\ref{thm3.1}, H. Busemann benefited  much from the fact that $G(n-1,E)$ is a convex set. That makes the proof in~\cite{busemann} very geometrical. For a slightly different approach to Theorem~\ref{thm3.1} which gives a more general statement, see~\cite{milman}.

\begin{rem}
 It is maybe worth noting that the density $\phi_{1,BH}$ is obviously convex in any finite dimensional normed vector space.
\end{rem}

Up to the late $1950$s, much effort was made to prove analogues of Theorem~\ref{thm3.1} for codimensions different from $1$ and $n-1$, without substantial progress, however. That culminated in the following problem formulated by H. Busemann and C. Petty (see~\cite{busemann3}, Problem 10).

\begin{probl}
\label{problem}
For which $m$ different from $1$ and $n-1$ are $BH$ densities of a normed vector space of dimension $n$ extendibly convex ?
\end{probl}

It was  more than $50$ years later that a new result of this type appeared in the literature. We mean the following theorem proved %by D.Burago and S.Ivanov 
in the article~\cite{buriv}.

\begin{theom}%{(D.Burago and S.Ivanov)}
	\label{m=2}
	Let  $E$ be a finite dimensional normed vector space with the unit ball $B$. Then the $BH$ density  $\phi_{2,BH,B}$ is extendibly convex.
	\end{theom}

It is interesting to remark that the proof of Theorem~\ref{m=2} in~\cite{buriv} is in a certain sense more analytical than geometrical. For an alternative proof of Theorem~\ref{m=2}, see~\cite{bernig}.

\bigskip

The following version of convexity of volume densities on Grassmann cones is strictly stronger than the extendible convexity. 
\begin{defin}{(Totally convex densities)}
	Let $E$ be a finite dimensional normed vector space of dimension $n$ and let $1\leq m\leq n$ be a natural number. A volume density $\phi$ on $G(m,E)$ is totally convex if for every $m$ dimensional linear subspace there exists a linear projection onto that subspace which does not increase the volume $\mathrm{Vol}_{\phi}$ corresponding to the density $\phi$.
		\end{defin}
		\begin{rem}
 Further in the text, we shall sometimes call such linear projections contractions.
\end{rem}

Indeed, it turns out that totally convex volume densities are extendibly convex but the reciprocal to this statement is false. Both these facts are explained in the article~\cite{busemann2}.

Euclidean nearest point projections have Lipschitz constant $1$. Thus, thanks to Eilenberg's inequality, Euclidean densities are totally convex. This argument, however, can not be generalized even to the case of the space $\ell^3_\infty$. Indeed, an easy computation shows that in this space, any linear projection onto the plane $\{(x_1,x_2,x_3):x_1+x_2+x_3=0\}$ has Lipschitz constant strictly bigger than $1$. However, a similar argument does show that one dimensional $BH$ densities $\phi_{1, BH}$ are totally convex in all finite dimensional normed vector spaces, thanks to the Hahn theorem.

Note that from the geometric measure theory and the calculus of variation points of view, the existence of area contractions is very useful in constructing compact supported solutions to the Plateau Problem, see~\cite{ambrsh}. Reader, interested in the Plateau Problem is welcomed to contact the articles~\cite{tdpchains},~\cite{tdpchains2}  %,~\cite{tdpbouf} 
 and~\cite{tdpapprox}, as well as the book~\cite{federer}.

		There exists an equivalent way of defining totally convex densities. We mean the following result discussed in~\cite{alvarezthompson}.
		\begin{prop}
		\label{alvar}
		Let $E$ and $m$ be as above. A function $\phi$ on $G(m,E)$ is totally convex if and only if it is extendibly convex and moreover the following holds. If $\Phi:\Lambda_m(E)\rightarrow \mathbb R$ is a convex extension of $\phi$, then through every point of the unit sphere $\mathbb S=\{x\in \Lambda_m(E): \Phi(x)=1\}$ there passes a supporting hyperplane of the form $\xi=1,$ where $\xi$ a simple $m$ vector in $\Lambda_m(E^*)$.
		\end{prop}

In the literature, the total convexity is also sometimes refereed to as the Gromov compressing property.
 		
Note that it can be derived from Theorem~\ref{thm3.1} that  in the codimension one case, each $BH$ density $\phi_{n-1, BH}$ is totally convex, see~\cite{busemann1}. The author of the present paper together with T. De Pauw have  used this observation in their article~\cite{ccepi} to prove the existence of compactly supported codimension one minimal $G$ chains (i.e. solutions to the corresponding version of the Plateau Problem) in finite dimensional normed vector spaces. However,  in the two-dimensional case in~\cite{ccepi} we had to use a weighted average of several projections (which we call in~\cite{ccepi} a density contractor) instead of just one projection. This allowed us to prove the existence of minimal $G$ chains of dimension two (alas, not necessary compactly supported) in the same context. As we shall see in this paper, there are some situations where sole projections are not available and  one indeed \textit{has to} consider those density contractors.
	
In the view of the discussion above and of Theorem~\ref{m=2}, it is natural to pose the following question: ``Are two dimensional $BH$ densities always totally convex ?'' Some time ago, the author of the current paper has been asked this question independently by Professors T. De Pauw and E. Stepanov. %, among whom was Professor E. Stepanov. 
 Note that the very same question had been already posed by H. Busemann and E. Strauss in~\cite{busemann2}. In more details, the authors in~\cite{busemann2} write that they do no know whether Busemann--Hausdorff densities are always totally convex in Minkowski spaces outside of dimension one and codimension one cases. In~\cite{busemann2}, by Minkowski spaces the authors mean  finite dimensional metric spaces.

%The following remark shows that there are densities which are convex but not totally convex.

	%\begin{rem}
%The following question was first formulated in paper~\cite{busemann2}: ``A it true that density is convex in Minkowski spaces ?'' By Minkowski space we mean a finite dimensional normed space with uniformly smooth and uniformly convex unit sphere. For $m\neq \{1,n-1\}$, in the very same paper~\cite{busemann2},
%it is possible to 
%the authors construct an example of a $m$ volume density which is extendibly convex, but not totally convex (see~\cite{busemann2}). This density, constructed in~\cite{busemann2}, is called the quadratic density. However, it is also proved in the same paper that this density \textit{is not a Minkowski space density}.
%$BH$ density}.
%	\end{rem}

\bigskip
We shall prove in the present paper that there are no two dimensional Hausdroff area contractions onto some planes in certain $n$ dimensional normed vector spaces for any $n\geq 4$. 
%the spaces $\ell_1^n$ for $n\geq 4$. 
This means that there are two dimensional $BH$ densities that are not totally convex.  We shall give a detailed proof in the case when $n=4$, the general case will follow from the four dimensional case. 
%in the same way. 
\begin{rem}
From now on, $X$ will stand for a certain fixed four dimensional (real) vector space, in which we fix a euclidean structure. The vectors of a fixed orthonormal basis $\mathcal B$ in $X$ will be denoted $e_1,e_2,e_3,e_4$.
\end{rem}

To this end, we start by noting that the unit ball of the space $\ell_1^4$, is given by the formula
$$\mathcal O:=\{x_1e_1+ x_2e_2 + x_3e_3 + x_4e_4\in X : |x_1|+|x_2|+|x_3|+|x_4|\leq 1\}.$$
As the matter of fact, this set is nothing but the unit four dimensional cross-polytope, also called regular octahedron. However, it will be convenient for us to consider a set $\mathcal C$ which is a rotated version of the unit cross-polytope. The former set is convex and symmetric, it is defined by $\mathcal C=f(\mathcal O)$, where $f$ is the linear mapping given by the following orthogonal matrix

$$
M:=
\begin{pmatrix}
\frac{1}{\sqrt{2}} & 0 & \frac{1}{2} & \frac{1}{2} \\
0 & \frac{1}{\sqrt{2}} & \frac{-1}{2} & \frac{1}{2} \\
\frac{1}{\sqrt{2}} & 0 & \frac{-1}{2} & \frac{-1}{2} \\
0 & \frac{-1}{\sqrt{2}} & \frac{-1}{2} & \frac{1}{2}
\end{pmatrix}
.
$$

It is easy to check that the set $\mathcal C$ likewise can be  defined as follows
\begin{multline}
\label{osnovnoemnozhestvo}
\mathcal C=\biggl\{y_1e_1+ y_2e_2 + y_3e_3 + y_4e_4\in X : \\
\Bigl|\frac{y_1}{\sqrt{2}}+\frac{y_3}{\sqrt{2}}\Bigr|+\Bigl|\frac{y_2}{\sqrt{2}}-\frac{y_4}{\sqrt{2}}\Bigr|+\Bigl|\frac{y_1-y_2-y_3-y_4}{2}\Bigr|+\Bigl|\frac{y_1+y_2-y_3+y_4}{2}\Bigr|\leq 1\biggr\}.
\end{multline}

We fix in the space $X$ the norm $\|\ldots\|$ given by the Minkowski functional of the set $\mathcal C$. Moreover, $\mathcal H^2_{\|\ldots\|}$ will stand for the Hausdorff measure induced by the norm  $\|\ldots\|$. 
%The intrinsic  ball of the space $\mathbb R^4$ of radius $r>0$ centered at a point $x\in X$ will be denoted $B(x,r)$.  
By $\mathcal H^2$ we shall mean the Euclidean Hausdorff measure of the space $X$, which we shall also sometimes call ``area''.

Here is the first main  result of this paper.

\begin{theorem}
\label{one}
In the space $X$, there is no linear contraction of the Hausdorff measure $\mathcal H^2_{\|\ldots\|}$ onto the two dimensional plane $\mathrm{span}(e_1,e_2)$.
\end{theorem}

\begin{rem}
As we have already mentioned, in~\cite{busemann2}, the authors constructed a density that is extendibly convex, but fails to be totally convex. This density belongs to a class of so-called quadratic densities. However, in the very same paper~\cite{busemann2}, the authors formulate that if a $BH$ density of some Minkowski space is quadratic, then the space is Euclidean. So, their counterexample does not cover the result of our Theorem~\ref{one}. 
\end{rem}

Obviously, the space $X$  is neither uniformly smooth, nor uniformly convex. Nevertheless, we can always approximate  $X$ by such a regular space, according to Lemma 2.3.2 in~\cite{her}. This yields uniformly smooth and uniformly convex finite dimensional normed vector spaces  with non-totally convex two dimensional $BH$ densities.

Of course, a similar result to that of our Theorem~\ref{one} holds if in the construction of the space $X$,  the four dimensional cross-polytope is replaced by an $n$ dimensional one with any $n\geq 4$. However, based on Theorem~\ref{one}, we prefer to derive for all $n\geq 4$ easier examples of $n$ dimensional normed vector spaces, in which the density $\phi_{2,BH}$ is not totally convex. These examples are based on the Cartesian product of certain spaces. We shall comment on these examples once again in the last remark of this paper.

The restriction on the dimension $n\geq 4$ is dictated on the one hand by the fact that if the dimension of the ambient space less than or equal to three, then by the already discussed all $BH$ densities are  totally convex. On the other hand, Proposition~\ref{alvar} indicates that the larger $n$, the harder it is to construct an example of a totally convex but not extendibly convex two dimensional density. This, together with the discussion in the last paragraph justifies our choice to concentrate on the case $n=4$ and to prove the lack of total convexity even in this case.

\bigskip

In order to get some intuition of the proof of Theorem~\ref{one}, we describe it here briefly. We give ourselves a linear mapping in $X$ which we assume to be an area contraction onto the plane $\mathrm{span}(e_1,e_2)$. Since it is a linear projection, it has a matrix relative to the fixed orthonormal base of a very particular form depending only on certain four real coefficients ($a$,$b$,$c$ and $d$). We want to show firstly that the only possibility is that it is the orthogonal linear projection, that is to say that all four coefficients are zero. Secondly, we will show that this orthogonal linear projection itself is not a contraction. This is the aim of the proof: we will have found a standardized space and a plane on which there is no linear contraction.

To show that the coefficients $a$,$b$,$c$ and $d$  are all zero, we fix $\varepsilon>0$ and associate it with several planes which depend on this $\varepsilon$. They are all ``very close'' to $\mathrm{span}(e_1,e_2)$. We reduce everything to purely Euclidean calculations, in particular to  
\begin{itemize}
\item calculations of Jacobians (which are norms of exterior products of matrices in the orthogonal basis) and 
\item calculations of areas (or approximations) of polygons in Euclidean planes.
\end{itemize}
The symmetry of the initial norm (that is to say of its unit ball) intervenes in the calculations to the extent that pairs of planes above are ``symmetrical'' with respect to each other, which means that in matrices which are involved in the calculations, some of the four coefficients (but not all) change sign.
An asymptotic expansion to order $2$ in $\varepsilon$ of a certain auxiliary coefficient $\lambda$ will allow us to conclude that, since the fixed linear mapping is a contraction, all four mentioned above coefficients  are zero (and this will be a subtle calculation). As a consequence, this will yield that if $\pi$ is an area contraction onto $\mathrm{span}(e_1,e_2)$, then $\pi$ is necessarily the orthogonal projection. 

It will remain afterwords to show only that the orthogonal projection onto $\mathrm{span}(e_1,e_2)$ is not an area contraction. There, we will use yet another plane that in turn will be quite  ``far away'' from $\mathrm{span}(e_1,e_2)$.

\bigskip

Our second main result describes a (probably) new situation, where the extendible convexity holds. In order to state our second main result, recall the notion of complex normed spaces. We call a normed vector space $(Y,\pmb{\|}\ldots\pmb{\|})$ complex if  the corresponding norm satisfies $\pmb{\|}\lambda (\cdot)\pmb{\|}=\pmb{\|}(\cdot)\pmb{\|}$ for all complex numbers $\lambda$. In other words, complex normed vector spaces possess unit balls with many rotational symmetries.
 
We are now in position to formulate our second theorem.

\begin{theorem}
	\label{two}
	In complex normed vector spaces, real codimension two $BH$ densities are extendibly convex.
	%Let $X$ be a complex normed space of dimension $2k$ for a natural $k\geq 2$.  Then the function  $\phi_{2k-2,BH}$ of the space $X$ is extendibly convex. 
\end{theorem}

We also have, in the notations of the article~\cite{ccepi}
 the following existence result for the Plateau Problem. 
\begin{theorem}
	\label{trii}
	Assume that: $(Y,\pmb{\|}\ldots\pmb{\|})$ is a complex normed vector space of dimension $2k$ for some natural $k\geq 2$, $(G,|\cdot|)$ is an Abelian normed locally compact White group,  $B\in \mathcal R_{2k-3}(Y,G)$ and $\partial B = 0$. It follows that the Plateau Problem
\begin{equation}
\label{plateau}
\begin{cases}
\text{minimize }  \mathit{M}_H(T) \\
\text{among all } T\in \mathcal R_{2k-2}(Y,G) \text{ such that } \partial T=B,
\end{cases}
\end{equation}
admits solutions.
	%In complex normed spaces, real codimension two $BH$ densities are extendibly convex.
	%Let $X$ be a complex normed space of dimension $2k$ for a natural $k\geq 2$.  Then the function  $\phi_{2k-2,BH}$ of the space $X$ is extendibly convex. 
\end{theorem}
Recall that $\mathcal R_{2k-3}(Y,G)$ and $\mathcal R_{2k-2}(Y,G)$ stand for the groups of rectifiable chains with coefficients in the group $G$ of dimensions $2k-3$ and $2k-2$ correspondingly and $\mathit{M}_H$ is the Hausdorff mass.

Both Theorems~\ref{two} and~\ref{trii} follow from Theorem~\ref{m=2}, a result that generalizes Theorem~\ref{thm3.1} for complex normed vector spaces,  and certain properties of the Hodge star operator. Note that in general Theorem~\ref{trii} does not follow from an extendtible convexity result. However, Theorem~\ref{trii} does follow from Theorem~\ref{two} in the cases where $G=\mathbb R$ and $G=\mathbb Z_2$ as explained in~\cite{buriv}.
\bigskip

Before proceeding to the proofs of our main results, let us discuss some problems that we leave open. First of all, we still do not know whether there are always compact solutions to the Plateau Problem in finite dimensional normed vector spaces in the context of $G$ chains of arbitrary dimension. We would like to stress that this question was discussed in~\cite{ambrsh} by L. Ambrosio and T. Schmidt. Second of all, it would be interesting to find out in which finite dimensional normed vector spaces the density $\phi_{2,BH}$ is totally convex. Third of all, we hope to generalize our Theorem~\ref{two} to all finite dimensional normed vector spaces. %Next, we would like to point out that the question of total convexity of Holmes--Thompson densities of hypermetric spaces was posed in~\cite{alvarezthompson} and, to the best of our knowledge, remains open. In connection with the last problem, it is interesting to remark that the normed space $\ell_1^4$, considered in this paper, is hypermetric.  
 Yet another problem to solve is to rule out (or to prove ?) the existence of \textit{nonlinear} contractions of the Hausdorff measure onto two dimensional planes. Let us finally mention once again   Problem~\ref{problem}. This question is at present far from being solved and is very intriguing. We conjecture the answer to Problem~\ref{problem} be negative and our hope is that the calculations, carried out in the proof of our first main result might be helpful to construct the desired counterexample. The author of the current paper plans to work on the first three of the aforementioned questions in the nearest future.

The rest of the paper is organized as follows. The second section is entirely devoted to the proof of Theorem~\ref{one}. The third section contains the proof of Theorems~\ref{two} and~\ref{trii}. In our fourth section which is the Appendix we give examples of spaces with non-totally convex $BH$ densities of any dimension greater than or equal to four.   

\subsection*{Acknowledgments}
The author is grateful to Thierry De Pauw and to Laurent Moonens for  helpful discussions. 

\section{Proof of Theorem~\ref{one}.}

\begin{proof}
First of all, note that any linear projection $\pi:X\rightarrow X$ onto $W_0:=\mathrm{span}(e_1,e_2)$ is given by a matrix of the following type
\begin{equation}
\begin{pmatrix}
\label{proj}
1 & 0 & a & b \\
0 & 1 & c & d \\
0 & 0 & 0 & 0 \\
0 & 0 & 0 & 0
\end{pmatrix}
,
\end{equation}
where $a,b,c,d$ are real numbers. Indeed, the elements in the last two lines of the matrix must be zeros, because the range of $\pi$ equals $W_0$ and in the left up corner there must be the $2\times 2$ identity submatrix, since $\pi$ is a projection.  

We shall prove the theorem by contradiction. Suppose, there exists a linear projection $\pi$ onto $W_0$, such that for any $V\in G(2,X)$ and for any two dimensional Euclidean disc $A$ in $V$ holds
\begin{equation}
\label{beforemain}
 \mathcal H^2_{\|\ldots\|}(\pi(A))\leq \mathcal H^2_{\|\ldots\|}(A).
\end{equation}
According to Busemann's theorem, see~\cite{thompson}, Theorem 7.3.5, we see that the inequality~\eqref{beforemain} reads

\begin{equation}
\label{main}
\frac{\mathcal H^2(\pi(A))}{\mathcal H^2(\mathcal C\cap W_0)}\leq \frac{\mathcal H^2(A)}{\mathcal H^2(\mathcal C\cap V)}.
\end{equation}

In order to reach a contradiction we shall further apply the inequality~\eqref{main} to nine specific planes $V_1,\ldots,V_9$, to be disclosed in the rest of this section. In order to get some intuition for this, we remark that the first eight of these planes will be infinitesimal perturbations of the plane $W_0$ and the ninth plane will be rather ``far away'' from $W_0$.

So, we fix an $\varepsilon>0$ and we first choose 
$$
V_1:= \mathrm{span}\biggl(\frac{e_1+\varepsilon e_3}{\sqrt{1+\varepsilon^2}},\frac{e_2+\varepsilon e_4}{\sqrt{1+\varepsilon^2}}\biggr) \text{ and } \; V_2:= \mathrm{span}\left(\frac{e_1-\varepsilon e_3}{\sqrt{1+\varepsilon^2}},\frac{e_2-\varepsilon e_4}{\sqrt{1+\varepsilon^2}}\right).
$$
Maybe, it is worth noting that the pairs of vectors 
$$\left(\frac{e_1+ \varepsilon e_3}{\sqrt{1+\varepsilon^2}},\frac{e_2+ \varepsilon e_4}{\sqrt{1+\varepsilon^2}}\right) \text{ and } \left(\frac{e_1- \varepsilon e_3}{\sqrt{1+\varepsilon^2}},\frac{e_2- \varepsilon e_4}{\sqrt{1+\varepsilon^2}}\right)$$ 
form orthonormal bases of the planes $V_1$ and $V_2$ correspondingly. Let $(f_1,f_2)$ denote the following orthonormal basis of $V_1$
$$
f_1:=\frac{(e_1+\varepsilon e_3)}{\sqrt{1+\varepsilon^2}} \; \text{ and } \; f_2:=\frac{(e_2+\varepsilon e_4)}{\sqrt{1+\varepsilon^2}}.$$ Next, we approximately calculate areas of the intersections of the set $\mathcal C$ with $V_1$ and $V_2$.

\begin{lemma}
\label{lem1}
 The following asymptotic relations hold, provided that positive $\varepsilon$ is small enough
 \begin{equation*}
   \mathcal H^2(\mathcal C\cap V_1)=\mathcal H^2(\mathcal C\cap V_2)\geq 12\sqrt{2}-16 + (272\sqrt{2}-384)\varepsilon^2 + o(\varepsilon^2).
 \end{equation*}
Moreover, holds
\begin{equation*}
 \mathcal H^2(\mathcal C\cap W_0)=\frac{8}{4+3\sqrt{2}}=12\sqrt{2}-16.
\end{equation*}
\end{lemma}

\begin{proof} A point that belongs to $V_1$ has the following coordinates in the basis $\mathcal B$
$$\left(\frac{x}{\sqrt{1+\varepsilon^2}},\frac{y}{\sqrt{1+\varepsilon^2}},\frac{\varepsilon x}{\sqrt{1+\varepsilon^2}},\frac{\varepsilon y}{\sqrt{1+\varepsilon^2}}\right),$$
where $x$ and $y$ are some real numbers. Recalling the definition of the set $\mathcal C$, we readily see that the coordinates of a point $xf_1+yf_2$ in $\mathcal C\cap V_1$ should fulfill the following inequality
\begin{multline}
 \label{vosmiug1}
  \frac{|x|(1+\varepsilon)}{\sqrt{2}\sqrt{1+\varepsilon^2}}+\frac{|y|(1-\varepsilon)}{\sqrt{2}\sqrt{1+\varepsilon^2}}+\frac{|x(1-\varepsilon)- y(1+\varepsilon)|}{2\sqrt{1+\varepsilon^2}}+ \frac{|x(1-\varepsilon)+ y(1+\varepsilon)|}{2\sqrt{1+\varepsilon^2}} \\ \leq 1.
\end{multline}

In other words, $\mathcal C\cap V_1$ is a planar set that consists of points with coordinates satisfying the inequality~\eqref{vosmiug1}. It is not difficult to prove that for $\varepsilon$ small enough, geometrically, the set $\mathcal C\cap V_1$ is an octagon.  However we shall not use this fact later on.  What we shall use, and what is obvious is that $\mathcal C\cap V_1$ symmetric with respect to the horizontal and the vertical axes. Let $S_1$ denote the area of the set $\mathcal C\cap V_1$. We want to minorize $S_1$, i.e. to find a lower bound on this quantity. To this end, we intersect $\mathcal C\cap V_1$ with the rays that are the halves of the planar lines defined by the equations $x=0$, $y=0$ and $x(1-\varepsilon)=y(1+\varepsilon)$ that belong to the first quadrant. In a moment, we shall find the intercepts that these rays have on the boundary of the set $\mathcal C\cap V_1$.

Thanks to the symmetries discussed in the previous paragraph we see that in order to estimate from below the area $S_1$, it suffices first to minorize the area of the intersection of $\mathcal C\cap V_1$ 
%of the octagon  $\mathcal C\cap V_1$ 
with the first quadrant and then to multiply the latter area by $4$. This intersection contains the quadrilateral $Q$ with vertices at the origin and at points $y_0f_2$, $x_0f_1$ and $x_1f_1+y_1f_2$, where $x_0, y_0, x_1$ and $y_1$ are positive numbers, to be calculated in a moment. Note that the last three of these four points are the intercepts discussed just above and that the area of $Q$ equals $(x_0y_1/2+x_1y_0/2)$.

Let us now express the numbers $x_0, y_0, x_1, y_1$ in terms of $\varepsilon$. Obviously, $x_0$ is the positive solution of the  equation
$$\sqrt{2}|x_0|(1+\varepsilon)+|x_0|(1-\varepsilon)+|x_0|(1-\varepsilon)=2\sqrt{1+\varepsilon^2},$$
from where  we get the following formula
$$x_0=\frac{2\sqrt{1+\varepsilon^2}}{2+\sqrt{2}-\varepsilon(2-\sqrt{2})}.$$ 
Analogously, we have that 
$$y_0=\frac{2\sqrt{1+\varepsilon^2}}{2+\sqrt{2}+\varepsilon(2-\sqrt{2})}.$$

%It is easy to see that the point $(x_1,y_1)$ satisfies the equation $x_1(1-\varepsilon)=y_1(1+\varepsilon).$ Hence holds
We further use the formula $x_1(1-\varepsilon)=y_1(1+\varepsilon)$ in order to calculate  the value $|y_1|$:
$$
|y_1|\biggl(\frac{\sqrt{2}(1+\varepsilon)^2}{1-\varepsilon}+\sqrt{2}(1-\varepsilon)+2(1+\varepsilon)\biggr)=2\sqrt{1+\varepsilon^2}.
$$
This readily allows us to derive the following expressions for the coordinates $x_1$ and $y_1$ 
$$
x_1=\frac{(1+\varepsilon)\sqrt{2}}{1+\sqrt{2}-\varepsilon^2 + \sqrt{2}\varepsilon^2}, \;  \; y_1=\frac{(1-\varepsilon)\sqrt{2}}{1+\sqrt{2}-\varepsilon^2 + \sqrt{2}\varepsilon^2}.
$$

We are now in position to establish an asymptotic inequality for the area $S_1$. We already know that $S_1\geq 2(x_0y_1+x_1y_0)$. This means that
\begin{equation} 
	\label{tvznak}
	\begin{split}
  S_1&\geq4\Bigl(\frac{x_0y_1}{2}+\frac{x_1y_0}{2}\Bigr)\\
&\geq\frac{4(1+\varepsilon^2)}{1+\sqrt{2}-\varepsilon^2+\sqrt{2}\varepsilon^2}\cdot\biggl(\frac{1-\varepsilon}{2+\sqrt{2}-\varepsilon(2-\sqrt{2})}+\frac{1+\varepsilon}{2+\sqrt{2}+\varepsilon(2-\sqrt{2})}\biggr)  \\
&=12\sqrt{2}-16+(272\sqrt{2}-384)\varepsilon^2+o(\varepsilon^2), 
\end{split}
\end{equation}
where the equality above is an easy computation of Taylor series. We do not detail its proof here.

To end up the proof of the first part of the lemma, it is now left to explain why holds
$$
\mathcal H^2(\mathcal C\cap V_1)=\mathcal H^2(\mathcal C\cap V_2).
$$
Indeed, this follows from the fact that the set $\mathcal C\cap V_2$ is defined in coordinates $(x,y)$ by the very same formula as the set $\mathcal C\cap V_1$ (we mean formula~\eqref{vosmiug1}).
Hence, the first statement of the lemma is proved. 

In order to prove the second statement of the lemma, i.e. the formula
$$
 \mathcal H^2(\mathcal C\cap W_0)=\frac{8}{4+3\sqrt{2}}=12\sqrt{2}-16,
$$
it suffices first to replace $\varepsilon$  by $0$ in calculations carried out in the proof of the first statement. This gives the inequality $\mathcal H^2(\mathcal C\cap W_0)\geq 12\sqrt{2}-16$. The reciprocal inequality is a simple exercise in planar geometry. It is left to the reader as an exercise.  Hence, the proof of the first lemma is finished.
\end{proof}

The next choices of planes are
$$
V_3:= \mathrm{span}\left(\frac{e_1-\varepsilon e_3}{\sqrt{1+\varepsilon^2}},\frac{e_2+\varepsilon e_4}{\sqrt{1+\varepsilon^2}}\right) \text{ and } \; V_4:= \mathrm{span}\left(\frac{e_1+\varepsilon e_3}{\sqrt{1+\varepsilon^2}},\frac{e_2-\varepsilon e_4}{\sqrt{1+\varepsilon^2}}\right).
$$
Let $(g_1,g_2)$ denote the following orthonormal basis of $V_1$
$$
g_1:=\frac{(e_1-\varepsilon e_3)}{\sqrt{1+\varepsilon^2}} \; \text{ and } \; g_2:=\frac{(e_2+\varepsilon e_4)}{\sqrt{1+\varepsilon^2}}.$$ As before, we calculate approximately the areas of intersections of the set $\mathcal C$ with the planes $V_3$ and $V_4$.

\begin{lemma}
\label{lem2}
The following asymptotic relations hold, provided that positive $\varepsilon$ is small enough
\begin{equation*}
   \mathcal H^2(\mathcal C\cap V_3)=\mathcal H^2(\mathcal C\cap V_4)\geq12\sqrt{2}-16 + (408\sqrt{2}-576)\varepsilon^2 + o(\varepsilon^2).
 \end{equation*}
\end{lemma}

\begin{proof} Note that a point that belongs to $V_3$ has the following coordinates in the basis $\mathcal B$
$$\biggl(\frac{x}{\sqrt{1+\varepsilon^2}},\frac{y}{\sqrt{1+\varepsilon^2}},\frac{-\varepsilon x}{\sqrt{1+\varepsilon^2}},\frac{\varepsilon y}{\sqrt{1+\varepsilon^2}}\biggr),$$
for some real numbers $x$ and $y$. Recalling the definition of $\mathcal C$, we see that  a point $xg_1+yg_2$ in $\mathcal C\cap V_3$ should fulfill the inequality
\begin{equation}
 \label{vosmiug2}
\sqrt{2}(1-\varepsilon)(|x|+|y|)+(1+\varepsilon)(|x+y|+|x-y|) \leq 2\sqrt{1+\varepsilon^2}.
\end{equation}
 Formula~\eqref{vosmiug2} defines the set $\mathcal C\cap V_3$ in the coordinates of the basis $(g_1,g_2)$. As in Lemma~\ref{lem1}, once $\varepsilon$ is small enough and positive, this set is an octagon, whose area will be denoted here $S_2$. It can be proved exactly as in the first lemma that this area is greater than or equal to four times the area of the quadrilateral with vertices at the origin and at the following points 
\begin{multline*}
\frac{\sqrt{2}\sqrt{1+\varepsilon^2}g_1}{\sqrt{2}+1-\varepsilon +\sqrt{2}\varepsilon}, \frac{\sqrt{2}\sqrt{1+\varepsilon^2}g_2}{\sqrt{2}+1-\varepsilon +\sqrt{2}\varepsilon}, \\
\frac{\sqrt{2}\sqrt{1+\varepsilon^2}g_1}{\sqrt{2}+2-2\varepsilon +\sqrt{2}\varepsilon}+\frac{\sqrt{2}\sqrt{1+\varepsilon^2}g_2}{\sqrt{2}+2-2\varepsilon +\sqrt{2}\varepsilon}.
\end{multline*}
These coordinates can be calculated exactly as in the proof of Lemma~\ref{lem1}. 

We further infer the estimates
\begin{align*}
 S_2 
&\geq8\biggl(\frac{1}{2}\cdot\frac{\sqrt{2}\sqrt{1+\varepsilon^2}}{\sqrt{2}+1-\varepsilon+\sqrt{2}\varepsilon}\cdot\frac{\sqrt{2}\sqrt{1+\varepsilon^2}}{\sqrt{2}+2-2\varepsilon+\sqrt{2}\varepsilon}\biggr) \\
 &=\frac{8(1+\varepsilon^2)}{(\sqrt{2}+1-\varepsilon+\sqrt{2}\varepsilon)(\sqrt{2}+2-2\varepsilon+\sqrt{2}\varepsilon)}\\
&=12\sqrt{2}-16+(408\sqrt{2}-576)\varepsilon^2+o(\varepsilon^2).
\end{align*}
These estimates are much the same as the bounds~\eqref{tvznak}. We leave them without
any comments. This together with the fact that the sets $\mathcal C\cap V_3$ and $\mathcal C\cap V_4$ obviously have the same area allow us to conclude that the second lemma is proved.
 \end{proof}

Next, armed with Lemmas~\ref{lem1} and~\ref{lem2}, we shall show that the two coefficients that are at the main diagonal of the $2\times2$ right up corner submatrix of the matrix~\eqref{proj} are equal to zero.

\begin{proop}
\label{prop1}
Let $\pi$ be a linear projection onto $\mathrm{span}(e_1,e_2)$ which is a contraction of the Hausdorff measure $\mathcal H^2_{\|\ldots\|}$. Then one necessarily has $a=d=0.$
\end{proop}
 \begin{proof} We claim that for any two dimensional Euclidean disc $A$ in $V_1$ holds
\begin{equation}
\label{lambda}
\mathcal H^2(\pi(A))=\lambda(V_1,W_0,\pi) \mathcal H^2(A),
\end{equation}
where $\lambda(V_1,W_0,\pi)$ is a positive coefficient, to be calculated in a moment. %This is an easy consequence of the coarea formula. %As the matter of fact, this coefficient  depends on $V_1, W_0$ and $\pi$, but $\pi$ itself depends only on $V_1$ and $W_0$.  

 Consider two auxiliary  linear mappings: the mapping $I$ defined by
%\begin{align*}
 % I: V_1 &\to \mathbb{R}^2  & J: \mathbb{R}^2  &\to W_0  \\
 % xf_1+yf_2 &\mapsto (x,y)^t,  & (x,y)^t  &\mapsto xe_1+ye_2.
%\end{align*}
 $$
\begin{array}{@{}c@{~}c@{~}c@{}}
I  : V_1 \;\;\;\;\;\; &\to  & \mathbb{R}^2 \\
\upin & & \upin \\
xf_1+yf_2 & \mapsto & (x,y)^t,
\end{array}
$$
where by $(x,y)^t$ we customly denote the transpose vector and the mapping $J$ defined by
 $$
\begin{array}{@{}c@{~}c@{~}c@{}}
J: \mathbb{R}^2 \;\;\;\;\;\; &\to  & W_0 \\
\upin & & \upin \\
(x,y)^t & \mapsto & xe_1+ye_2.
\end{array}
 $$

Introduce an auxiliary matrix
$$
N:=
\begin{pmatrix}
1+\frac{\varepsilon a}{\sqrt{1+\varepsilon^2}} & \frac{\varepsilon b}{\sqrt{1+\varepsilon^2}} \\
\frac{\varepsilon c}{\sqrt{1+\varepsilon^2}} & 1+\frac{\varepsilon d}{\sqrt{1+\varepsilon^2}} \\
\end{pmatrix}
,
$$
and the corresponding linear mapping $P:  \mathbb{R}^2 \to \mathbb{R}^2$ that sends a vector $\nu\in\mathbb R^2$ to the vector $N\nu\in \mathbb R^2$. It is easy to verify that for $v\in V_1$ holds $J(P(I(v)))=\pi(v).$ Since $J$ is an isometry, we infer the formula $$\mathcal H^2(\pi(A))=\mathcal H^2(P(I(A)).$$ The area of the linear mapping $P$ can be easily calculated explicitly. Thanks to this observation and  to the fact that $I$ is an isometry, we have that 
\begin{align*}
\mathcal H^2(\pi(A))&=|\mathrm{det}(N)|\mathcal H^2(A) \\
&=\left|\left(1+\frac{\varepsilon a}{\sqrt{1+\varepsilon^2}}\right)\cdot\left(1+\frac{\varepsilon d}{\sqrt{1+\varepsilon^2}}\right)-\frac{\varepsilon^2 bc}{1+\varepsilon^2}\right|\cdot\mathcal H^2(A),
\end{align*}
and our claim follows.

Therefore, taking into account the first lemma we see that line~\eqref{main} applied to $V_1$ reads as follows
$$
1+\frac{\varepsilon(a+d)}{\sqrt{1+\varepsilon^2}}+\frac{\varepsilon^2(ad-bc)}{1+\varepsilon^2}\leq\frac{12\sqrt{2}-16}{12\sqrt{2}-16+(272\sqrt{2}-304)\varepsilon^2+o(\varepsilon^2)}.
$$
 From here, we infer that for $\varepsilon$ small enough holds $a+d\leq \varepsilon c$, where $c$ is some real number. Letting $\varepsilon$ tend to $0$, we arrive at the inequality $a+d\leq 0$.

Further, we apply the same reasoning as that in the three paragraphs just above, but now to the plane $V_2$. This leads to the following chain of equalities and inequalities
\begin{align*}
&\det
\begin{pmatrix}
1-\frac{\varepsilon a}{\sqrt{1+\varepsilon^2}} & \frac{\varepsilon b}{\sqrt{1+\varepsilon^2}} \\
\frac{\varepsilon c}{\sqrt{1+\varepsilon^2}} & 1-\frac{\varepsilon d}{\sqrt{1+\varepsilon^2}} \\
\end{pmatrix}\\
&= \left(1-\frac{\varepsilon a}{\sqrt{1+\varepsilon^2}}\right)\cdot\left(1-\frac{\varepsilon d}{\sqrt{1+\varepsilon^2}}\right)-\frac{\varepsilon^2 bc}{1+\varepsilon^2} \\
&\leq \frac{12\sqrt{2}-16}{12\sqrt{2}-16 + (408\sqrt{2}-576)\varepsilon^2 +o(\varepsilon^2)}.
\end{align*}
Letting $\varepsilon$ tend to zero, we conclude that $-(a+d)\leq 0$ and in turn that $a+d=0$. 

Finally, armed with Lemma~\ref{lem2}, we are now able to apply the same reasoning as in the previous part of the proof of Proposition~\ref{prop1}, but now to the planes $V_3$ and $V_4$. This leads to the inequalities $a-d\leq 0$ and $-(a-d)\leq 0$ respectively. 

Hence, $a=d=0$ and the current proposition is thus proved.
\end{proof}

Further four choices of planes are like this:
\begin{align*}
V_5&:= \mathrm{span}\left(\frac{e_1+\varepsilon e_4}{\sqrt{1+\varepsilon^2}},\frac{e_2+\varepsilon e_3}{\sqrt{1+\varepsilon^2}}\right), \; \; V_6:= \mathrm{span}\left(\frac{e_1-\varepsilon e_4}{\sqrt{1+\varepsilon^2}},\frac{e_2-\varepsilon e_3}{\sqrt{1+\varepsilon^2}}\right), \\
V_7&:= \mathrm{span}\left(\frac{e_1-\varepsilon e_4}{\sqrt{1+\varepsilon^2}},\frac{e_2+\varepsilon e_3}{\sqrt{1+\varepsilon^2}}\right), \; \; V_8:= \mathrm{span}\left(\frac{e_1+\varepsilon e_4}{\sqrt{1+\varepsilon^2}},\frac{e_2-\varepsilon e_3}{\sqrt{1+\varepsilon^2}}\right).
\end{align*}
Acting as before, one can obtain approximate values of the area of the intersections of the set $\mathcal C$ with the planes $V_5,\ldots,V_8.$

 \begin{lemma}
There exist two positive absolute constants $c_1$ and $c_2$ such that following asymptotic relations hold, provided that positive $\varepsilon$ is small enough
\begin{align*}
    \mathcal H^2(\mathcal C\cap V_5)&=\mathcal H^2(\mathcal C\cap V_6)\geq12\sqrt{2}-16 + c_1\varepsilon^2 + o(\varepsilon^2)\\
    \mathcal H^2(\mathcal C\cap V_7)&=\mathcal H^2(\mathcal C\cap V_8)\geq12\sqrt{2}-16 + c_2\varepsilon^2 + o(\varepsilon^2).
   \end{align*}
\end{lemma}

\begin{proof} The proof of this statement is similar to that of Lemma~\ref{lem1} and Lemma~\ref{lem2}, so it is not detailed here.
\end{proof}

With the help of the third lemma, we are now in position to show that all coefficients of the $2\times2$ right up corner submatrix of the matrix~\eqref{proj} are equal to zero.

\begin{proop}
\label{prop2}
Let $\pi$ be a linear projection defined onto $\mathrm{span}(e_1,e_2)$ which is a contraction of the Hausdorff measure $\mathcal H^2_{\|\ldots\|}$. Then one necessarily has $a=b=c=d=0.$
\end{proop}

\begin{proof} By Proposition~\ref{prop1} we have that $a=d=0$, and the proof of the fact that $b=c=0$ is similar to that of Proposition~\ref{prop1}, so it is not detailed here and is left to the reader as an easy exercise.
\end{proof}

\begin{rem}
Note that Propositions~\ref{prop1} and~\ref{prop2} show that the projection $\pi$ must be orthogonal.
\end{rem}

How to conclude then as for the proof of  Theorem~\ref{one} ? Consider the final special choice of a plane. Put 
$$
V_9:=\mathrm{span}\biggl(\biggl(\frac{1}{\sqrt{2}},0,\frac{1}{\sqrt{2}},0\biggr),\biggl(0,\frac{1}{\sqrt{2}},0,-\frac{1}{\sqrt{2}}\biggr)\biggr).
$$
We claim that $\mathcal H^2(\mathcal C\cap V_9)=2$. Indeed, this results from the following two observations 
$$f(e_1)=\biggl(\frac{1}{\sqrt{2}},0,\frac{1}{\sqrt{2}},0\biggr) \; \text{ and } \; f(e_2)=\biggl(0,\frac{1}{\sqrt{2}},0,-\frac{1}{\sqrt{2}}\biggr)$$
and from the fact that obviously $\mathcal H^2(\mathcal O\cap \mathrm{span}(e_1,e_2))=2$. 

Propositions~\ref{prop1} and~\ref{prop2} tell us that the projection $\pi$ is defined by the following matrix
\begin{equation*}
\begin{pmatrix}
1 & 0 & 0 & 0 \\
0 & 1 & 0 & 0 \\
0 & 0 & 0 & 0 \\
0 & 0 & 0 & 0
\end{pmatrix}
.
\end{equation*} 
From here, acting as in the proof of formula~\eqref{lambda} one can easily deduce that for any two dimensional Euclidean disc $A$ in $V_9$ holds
%Let finally $\lambda(V_9,W_0,\pi)$ denote the corresponding coefficient defined exactly as in formula. Note that 
$$\mathcal H^2(\pi(A))= \lambda(V_9,W_0,\pi) \mathcal H^2(A),$$
where
$$
\lambda(V_9,W_0,\pi)=
\det
\begin{pmatrix}
\frac{1}{\sqrt{2}} & 0 \\
0 & \frac{1}{\sqrt{2}} \\
\end{pmatrix}
=\frac{1}{2}
.
$$

Inequality~\eqref{main} applied to $V_9$ thus implies $1/2\leq (12\sqrt{2}-16)/2,$
which is a contradiction, since an easy computation shows that $(12\sqrt{2}-16)/2<0.49$. Hence, Theorem~\ref{one}  is proved.

\end{proof}

\begin{rem}
It is interesting to mention that the plane $\mathrm{span}(e_1,e_2)$ is the section of the least two dimensional area of the set $\mathcal C$. Here, a ``section'' means an ``intersection with a two dimensional linear subspace''.  This has been  recently proved in~\cite{nazar} by elementary methods. It is plausible that the calculations, carried out in our proof above, namely in the proofs of our three lemmas, might provide an alternative proof of this fact. 
\end{rem}

\section{Extendible convexity in codimension two case in complex convex spaces}

Let us now show how to deduce Theorem~\ref{two} from Theorem~\ref{m=2} and a Busemann type result from paper~\cite{kold}.

\begin{proof}
So, let $Y$ be a complex normed vector space of dimension $2k$ with an integer $k>1$ and let $\mathbb B$ denote the unit ball of the space $Y$. Fix $\omega\in G(2k-2,X)$ and recall that $\star:\Lambda_2(Y)\rightarrow \Lambda_{2k-2}(Y)$ denotes the Hodge star operator (see e.g. in~\cite{flan} for a definition). According to~\cite{kold}, Theorem 7.1, there exists a convex symmetric body $K\subset Y$ such that for all $\omega\in G(2k-2,Y)$ holds
$$
\mathcal H^{2k-2}(\mathbb B\cap \mathrm{span}(\omega))=\mathcal H^{2}(K\cap \mathrm{span}(\star\omega)).
$$
 Thanks to the last line above we have
\begin{align*}
	\phi_{2k-2,BH,\mathbb B}(\omega)&=\frac{\mathbold{\alpha}(2k-2)|\omega|_2}{\mathcal H^{2k-2}(\mathbb B\cap \mathrm{span}(\omega))}\\
	&=\frac{\mathbold{\alpha}(2k-2)|\omega|_2}{\mathcal H^{2}(K\cap \mathrm{span}(\star\omega))}\\
	&=\sqrt{\frac{2!}{(2k-2)!}}\frac{\mathbold{\alpha}(2k-2)|\mathord\star\omega|_2}{\mathcal H^{2}(K\cap \mathrm{span}(\star\omega))}\\
	&=\sqrt{\frac{2!}{(2k-2)!}}\frac{\mathbold{\alpha}(2k-2)}{\mathbold{\alpha}(2)}\phi_{2,BH,K}(\star\omega),
	\end{align*}
	where the penultimate equality above is an easy consequence of Example 1.12 in~\cite{mark}.

The function $\phi_{2,BH,K}$  is extendibly convex thanks to Theorem~\ref{m=2}. As a consequence of the linearity of the Hodge star, we conclude that the function $\phi_{2k-2,BH,\mathbb B}$ is extendibly convex.
\end{proof}

\bigskip

Let us now prove Theorem~\ref{trii}. We would like to stress that we adopt notations of the article~\cite{ccepi}.

\begin{proof}
 As in Theorem~\ref{two},  we let $\mathbb B$ denote the unit ball of the space $Y$. According to Theorems 4.3  and 3.1 in~\cite{ccepi}, it suffices to prove the following triangular inequality for cycles.
\begin{proop}
\label{triangle}
For every integer $\kappa\geq 2$, every oriented simplexes $\sigma_1,\ldots,\sigma_\kappa$ of dimension $2k-2$ in $Y$ and every $g_1,\ldots,g_\kappa \in G$, if $\partial \sum_{i=1}^\kappa g_i [\sigma_i]=0$, then
$$
\mathit{M}_H (g_1 [\sigma_1]) \leq \sum_{i=2}^\kappa \mathit{M}_H (g_i [\sigma_i]).
$$
\end{proop}
\begin{proof}
We know that each $\sigma_i$ belongs to some affine $2k-2$ dimensional subspace in $Y$, i.e. that for all natural $i$ from $1$ to $\kappa$ holds $\sigma_i\subset \mathrm{span}(\omega_i)+h_i$ for some decomposable $2k-2$ vectors $\omega_i$ of unit norm  and some  $h_i\in Y$. Taking into account the definition of the Hausdorff mass (see~\cite{ccepi}) and the Busemann theorem (see~\cite{thompson}, Theorem 7.3.5) we see that what we need to prove is the following inequality
\begin{equation}
\label{hodge-1}
|g_1|\frac{\mathcal H^{2k-2}(\sigma_1)}{\mathcal H^{2k-2}(\mathbb B\cap \mathrm{span}(\omega_1))}\leq \sum_{i=2}^{\kappa}|g_i|\frac{\mathcal H^{2k-2}(\sigma_i)}{\mathcal H^{2k-2}(\mathbb B\cap \mathrm{span}(\omega_i))}.
\end{equation}

Let $K$ be as in the proof Theorem~\ref{two}. We apply Theorem 7.1 from~\cite{kold} to find that
\begin{equation}
\label{hodge-0.5}
\sum_{i=2}^{\kappa}|g_i|\frac{\mathcal H^{2k-2}(\sigma_i)}{\mathcal H^{2k-2}(\mathbb B\cap \mathrm{span}(\omega_i))}=\sum_{i=2}^{\kappa}|g_i|\frac{\mathcal H^{2k-2}(\sigma_i)}{\mathcal H^{2}(K\cap \mathrm{span}(\mathord\star\omega_i))}.
\end{equation}
Note that by Theorem 1.8 in~\cite{mark}, each  $\mathord\star\omega_i$ is a decomposable two vector. Choose an orthonormal basis $(v_{1,i},v_{2,i})$ of the plane $\mathrm{span}(\mathord\star\omega_i)$.

Fix $\epsilon>0$ and approximate  using Lemma 2.3.2 in~\cite{her} the convex symmetric body $K$ by a convex symmetric polyhedron $K_\epsilon$, which is the unit ball of some crystalline norm in $Y$ in a way that
$$
\max_{i\in [1,\ldots,\kappa]}|\mathcal H^2(K_\epsilon\cap \mathrm{span}(\omega_i))-\mathcal H^2(K\cap \mathrm{span}(\omega_i))|<\epsilon.
$$
As a consequence, we have
\begin{equation}
\label{hodge-0.01}
\sum_{i=2}^{\kappa}|g_i|\frac{\mathcal H^{2k-2}(\sigma_i)}{\mathcal H^{2}(K\cap \mathrm{span}(\mathord\star\omega_i))}\geq \sum_{i=2}^{\kappa}|g_i|\frac{\mathcal H^{2k-2}(\sigma_i)}{\mathcal H^{2}(K_\epsilon\cap \mathrm{span}(\mathord\star\omega_i))} -\widehat{\epsilon},
\end{equation}
where $\widehat{\epsilon}>0$ tends to zero, when $\epsilon$ tends to zero.

 Taking into account Theorem 5.6 in~\cite{ccepi}, we see that 
\begin{equation}
	\begin{split}
\label{hodge-0.1}
&\sum_{i=2}^{\kappa}|g_i|\frac{\mathcal H^{2k-2}(\sigma_i)}{\mathcal H^{2}(K_\epsilon\cap \mathrm{span}(\mathord\star\omega_i))}\\
&\geq \sum_{i=2}^{\kappa}|g_i|\mathcal H^{2k-2}(\sigma_i)\sum_{1\leq r<l\leq p} \lambda_r \lambda_l|\langle \alpha_r\wedge \alpha_l,v_{1,i}\wedge v_{2,i} \rangle|,
\end{split}
\end{equation}
for certain natural $p$, real $\{\lambda_r\}_{r=1}^p$ and $\alpha_r\in Y$ for all natural $r$ from $1$ to $p$. Based on Example 1.12 in~\cite{mark} we deduce that 
\begin{equation} 
	\label{hodge0}
	\begin{split}
&\sum_{i=2}^{\kappa}|g_i|\mathcal H^{2k-2}(\sigma_i)\sum_{1\leq r<l\leq p} \lambda_r \lambda_l|\langle \alpha_r\wedge \alpha_l, v_{1,i}\wedge v_{2,i} \rangle|\\
&=\sum_{i=2}^{\kappa}|g_i|\mathcal H^{2k-2}(\sigma_i)\frac{2!}{(2k-2)!}\sum_{1\leq r<l\leq p} \lambda_r \lambda_l|\langle \mathord\star (\alpha_r\wedge \alpha_l),\omega_i \rangle|\\
&=\sum_{1\leq r<l\leq p} \lambda_r \lambda_l\frac{2!}{(2k-2)!}\sum_{i=2}^{\kappa}|g_i|\mathcal H^{2k-2}(\pi_{r,l}\sigma_i),
\end{split}
\end{equation}
where $\pi_{r,l}$ are certain linear projections from $Y$ onto $\mathrm{span}(\omega_1)$, very much the same  as those in the proof of Theorem 5.9 in~\cite{ccepi}. Using the Euclidean triangular inequality for cycles  (see e.g.~\cite{ccepi}, Theorem 5.3) and once again Example 1.12 in~\cite{mark} we infer that
\begin{equation}
\label{hodge1}
\begin{split}
\sum_{1\leq r<l\leq p}&\lambda_r \lambda_l\frac{2!}{(2k-2)!}\sum_{i=2}^{\kappa}|g_i|\mathcal H^{2k-2}(\pi_{r,l}\sigma_i) \\
&\geq\sum_{1\leq r<l\leq p}\lambda_r \lambda_l\frac{2!}{(2k-2)!}|g_1| \mathcal H^{2k-2}(\pi_{r,l}\sigma_1)\\
&= |g_1| \mathcal H^{2k-2}(\sigma_1)\sum_{1\leq r<l\leq p} \lambda_r\lambda_l |\langle  \alpha_r\wedge \alpha_l, v_{1,1}\wedge v_{2,1} \rangle|.
\end{split}
\end{equation}

By virtue of the equality case in Theorem 5.9 from~\cite{ccepi} and of Theorem 7.1 from~\cite{kold} we conclude that
\begin{equation} 
	\label{hodge2}
	\begin{split}
&|g_1| \mathcal H^{2k-2}(\sigma_1)\sum_{1\leq r<l\leq p} \lambda_r\lambda_l |\langle  \alpha_r\wedge \alpha_l, v_{1,1}\wedge v_{2,1} \rangle|\\
&=|g_1| \frac{\mathcal H^{2k-2}(\sigma_1)}{\mathcal H^{2}(K_\epsilon\cap \mathrm{span}(\mathord\star\omega_1))}\\
&\geq |g_1| \frac{\mathcal H^{2k-2}(\sigma_1)}{\mathcal H^{2}(K\cap \mathrm{span}(\mathord\star\omega_1))}-\widetilde{\epsilon}\\
&=|g_1| \frac{\mathcal H^{2k-2}(\sigma_1)}{\mathcal H^{2k-2}(\mathbb B\cap \mathrm{span}(\omega_1))}-\widetilde{\epsilon},
\end{split}
\end{equation}
where $\widetilde{\epsilon}$ tends to zero, when $\epsilon$ tends to zero.
Here above, in the last equality we have used Theorem 7.1 in~\cite{kold} once again. 

Getting the inequalities~\eqref{hodge-0.5},~\eqref{hodge-0.01},~\eqref{hodge-0.1},~\eqref{hodge0},~\eqref{hodge1} and~\eqref{hodge2} together and letting $\epsilon$ tend to zero yields the triangle inequality~\eqref{hodge-1}.
\end{proof}

Thus, Theorem~\ref{trii} follows from Proposition~\ref{triangle}, by virtue of Theorem 4.3  and Theorem 3.1 in~\cite{ccepi}.
\end{proof}

\section{Appendix}

Let us now give for each natural $n> 4$ an example of an $n$ dimensional   normed vector space $Y$ in which the two dimensional $BH$ density is not totally convex. To this end, recall the set $\mathcal C$ (see formula~\eqref{osnovnoemnozhestvo}) and let $\mathrm{C}$ denote the body $\mathcal C \times B_2^{n-4}$. Here, $\times$ stands for the Cartesian product and $B_2^{n-4}$ for the Euclidean unit ball of dimension $n-4$. Consider the $n$ dimensional normed vector space %Let $\widetilde{X}$ denote the space 
$Y$, whose norm $\pmb{|}\ldots\pmb{|}$ is defined by the Minkowski functional of $\mathrm{C}$. 
\begin{proop}
In the space $Y$, there is no linear contraction of the Hausdorff measure $\mathcal H^2_{\pmb{|}\ldots\pmb{|}}$ onto the two dimensional plane $\mathrm{span}(e_1,e_2)$.
\end{proop}
\begin{proof}
Suppose the controrary, i.e. suppose that in this space 
%$\widetilde{X}$ 
there was a linear projection onto the plane $\mathrm{span}(e_1,e_2)$ which is a contraction of the Hausdorff measure $\mathcal H^2_{\pmb{|}\ldots\pmb{|}}$. %Here, $e_1,\ldots, e_n$ stands for the standard basis of $\mathbb R^n$.  
 Denote this projection $\Pi$. Let $X$ and $W_0$ be as in the proof of Theorem~\ref{one}.  Thus, any plane $V$ in $X$ and for any two dimensional Euclidean disc $A$ in $V$  would satisfy
\begin{equation*}
\mathcal H^2_{\pmb{|}\ldots\pmb{|}}(\Pi(A))\leq \mathcal H^2_{\pmb{|}\ldots\pmb{|}}(A).
\end{equation*}

But then, by the construction of the norm $\pmb{|}\ldots\pmb{|}$ we would also have the following inequality
\begin{equation*}
\mathcal H^2_{\|\ldots\|}(\Pi\restr{W_0}(A))\leq \mathcal H^2_{\|\ldots\|}(A),
\end{equation*}
which contradicts Theorem~\ref{one}.
\end{proof}
\begin{rem}
It is plausible that this construction allows to generalize --- once it is constructed --- a counterexample to Problem~\ref{problem} to all higher dimensions and codimensions.
\end{rem}
 % of any dimension greater than or equal to four.

%Let us finally mention one case, when the convexity is known.
%\begin{proop}
%Let $n$ be a natural number strictly bigger than $1$ and let $\mathcal S$ be a $2n$ dimensional normed vector space with complex convex unit ball. Then the function $\phi_{2n-2,BH}$ of the space $\mathcal S$ is extendibly convex.
%\end{proop}
%\begin{proof}
% 
% DOPISAT'
% 
% \end{proof}

\renewcommand{\refname}{References}

\end{document}